\documentclass{amsart}
\usepackage[T1]{fontenc}
\usepackage{fullpage}
\usepackage[top=2cm, bottom=4.5cm, left=2.5cm, right=2.5cm]{geometry}
\usepackage{amsmath,amsthm,amsfonts,amssymb,amscd}
\usepackage{lastpage}
\usepackage{enumerate}
\usepackage{fancyhdr}
\usepackage{mathrsfs}
\usepackage{xcolor}
\usepackage{graphicx}
\usepackage{listings}
\usepackage{hyperref}
\usepackage{tikz}
\usepackage{tikz-cd}
\usepackage{commath}
\usepackage{bbm}
\usepackage{multirow}
\usepackage{multicol}
\usepackage{enumitem}
\usepackage{pgfplots}
\usepgfplotslibrary{polar}
\usepgflibrary{shapes.geometric}
\usetikzlibrary{calc}
\usepgfplotslibrary{fillbetween}
\usetikzlibrary{patterns}
\usetikzlibrary{calc}
\usetikzlibrary{decorations.pathmorphing}
\usepackage{stmaryrd}
\usepackage{placeins}
\usepackage{amsrefs}

\usepackage{url}
\usepackage{wasysym}
\usepackage{tabularx}
\usepackage{makecell}
\usepackage{cleveref}

\newcommand{\p}{\partial}

\newcommand{\R}{\mathbb{R}}

\newcommand{\ra}{\rightarrow}

\newcommand{\ad}{\mathrm{ad}}
\newtheorem{theorem}[subsubsection]{Theorem}
\newtheorem{corollary}[subsubsection]{Corollary}
\newtheorem{lemma}[subsubsection]{Lemma}
\theoremstyle{definition}
\newtheorem{definition}[subsubsection]{Definition}
\newtheorem{prop}[subsubsection]{Proposition}
\newtheorem*{Acknow}{Acknowledgements}
\newtheorem{remark}[subsubsection]{Remark}

\newtheorem*{definition*}{Definition}
\newtheorem*{theorem*}{Theorem}
\newtheorem*{corollary*}{Corollary}
\newtheorem*{problem*}{Problem}
\newtheorem*{question*}{Question}
\newcommand{\A}{\mathcal{A}}

\theoremstyle{definition}
\newtheorem{exmp}[subsubsection]{Example}

\DeclareMathOperator{\End}{End}

\DeclareMathOperator{\im}{im}

\newcommand{\vdotrows}[2][-15pt]{%
  \multirow{#2}{*}{%
    \vbox {%
      \baselineskip = \dimexpr 2pt\relax 
      \multiply\baselineskip by #2\relax
      \advance\baselineskip by 2pt\relax
      \lineskiplimit = 0pt\relax
      \kern 6pt\relax
      \vskip #1\relax%
      \hbox {.}\hbox {.}\hbox {.}}%
  }%
}

\title[Algebroid Desingularizable Poisson Structures]{Algebroid Desingularizable Poisson Structures}

\author{Shane Rankin}
\address{Department of Mathematics, 900 University Ave, Riverside, CA 92507}
\email{srank007@ucr.edu}

\pgfplotsset{compat=1.18}
\begin{document}

\begin{abstract}
 We introduce \emph{algebroid desingularizable Poisson manifolds}, those Poisson manifolds induced by a symplectic Lie algebroid with an almost-injective anchor. This class coincides with that of $E$-symplectic manifolds, and contains log-symplectic manifolds and $b^m$-symplectic manifolds. We give an infinitesimal obstruction to desingularizability: at a zero of the Poisson structure the isotropy Lie algebra must contain an abelian ideal of at least half the dimension of the manifold, so no Poisson structure with a semisimple isotropy Lie algebra can be desingularized. We then characterize the linear case, showing that the dual of a real $n$-dimensional Lie algebra $\mathfrak{g}$, equipped with the KKS Poisson structure, is desingularizable if and only if $\mathfrak{g}$ possesses an abelian ideal of dimension $n-r$, where $2r$ is the maximal coadjoint orbit dimension. In this case, we construct a desingularizing algebroid explicitly.
\end{abstract}
\maketitle

\section{Introduction and Results}
Generically symplectic Poisson structures arise naturally across mathematical physics. The dynamics of pseudo-Riemannian manifolds \cite{MR2518642}, Painlevé equations \cite{MR636469}, and the restricted three-body problem \cite{MR656211} produce Poisson structures that are necessarily degenerate, so the bivector is not inverse to any symplectic form. In favorable situations, however, one can construct a new vector bundle whose sections carry a Lie bracket analogous to the Lie bracket on vector fields, on which a $2$-form playing the role of the symplectic form dual to the chosen bivector exists as a smooth, closed, nondegenerate element. Such a bundle is a \emph{symplectic Lie algebroid}.
Symplectic Lie algebroids were introduced by Nest and Tsygan \cite{nest1999deformationssymplecticliealgebroids}, who showed they admit Fedosov deformation quantization and canonically induce a Poisson structure on the base manifold. \par 
Subsequent work has shown that many important families of Poisson manifolds arise this way: log-symplectic manifolds, and their generalizations, $b^m$-Poisson manifolds, both admit a symplectic Lie algebroid realizing their Poisson structure, see \cites{M_rcu__2014, Guillemin_2014}. In the physics literature, these structures have been studied under the name quasi-symplectic manifolds in connection with deformation quantization and BRST theory \cite{Lyakhovich_2006}. More recently, deformation quantization of the underlying Poisson structure via Toeplitz operators on the inducing algebroid has been developed in \cite{Toeplitz}. In \cite{bischoff2026poissongeometrytruncatedpolynomials}, Bischoff and Witte study the case in which the anchor map of the algebroid is generically an isomorphism onto the tangent bundle, and extract substantial information about the underlying Poisson structure from the algebroid's symplectic geometry. 
\par The common thread running through all of these examples is that the anchor map is injective on an open dense subset of the base manifold. Such algebroids are called almost-injective, and we call a Poisson manifold \emph{algebroid desingularizable} or just \emph{desingularizable} if it is induced by a symplectic Lie algebroid with an almost-injective anchor. This notion is equivalent to that of an $E$-symplectic manifold as studied in \cite{garmendia2025estructuresregularpoissonmanifolds,MR4236806}, and this raises a natural question: which Poisson manifolds are desingularizable? On a symplectic leaf of maximal dimension, such an algebroid's anchor is an isomorphism onto the tangent bundle of the leaf, and so obstructions lie at singular points. If $x \in M$ is such that $\pi_x$ vanishes, then $T_x^*M$ inherits the structure of a Lie algebra from the Poisson bracket, and we have the following:
\begin{theorem*}
      Let $(M,\pi)$ be a Poisson manifold of dimension $n$, and let $x \in M$ be a point at which $\pi_x$ vanishes. If $(M,\pi)$ is desingularizable, then $(T^*_xM, [\cdot,\cdot]_{\pi})$ has an abelian ideal $\mathfrak{a}$ satisfying $\dim(\mathfrak{a}) \geq \frac{n}{2}$. In particular, if this Lie algebra is semisimple, then $(M,\pi)$ is not desingularizable.
\end{theorem*}
For linear Poisson structures, that is, for duals of finite-dimensional real Lie algebras equipped with the Kirillov-Kostant-Souriau Poisson structure, this obstruction is sharp in the following sense:
\begin{theorem*}
     Let $\mathfrak{g}$ be a real finite-dimensional Lie algebra of dimension $n$, with $2r=\max_{\mathcal{O} \in \mathcal{L}}\{\dim(\mathcal{O})\}$ where $\mathcal{L}$ is the set of coadjoint orbits. Then $\mathfrak{g}^*$ is desingularizable if and only if there exists an abelian ideal $\mathfrak{a}$ of $\mathfrak{g}$ of dimension $n-r$.
\end{theorem*}
From this, we have the following corollaries:
\begin{corollary*}
     If $\mathfrak{g}$ is a real finite-dimensional, non-abelian reductive Lie algebra, then $\mathfrak{g}^*$ is not desingularizable.
\end{corollary*}
\begin{corollary*}
     If $\mathfrak{g}$ is a real Lie algebra of dimension $3$ or less, then $\mathfrak{g}^*$ is desingularizable if and only if $\mathfrak{g}$ is not semisimple.
\end{corollary*}
\begin{corollary*}
     Suppose that $\mathfrak{g}$ is a real $n$-dimensional $2$-step nilpotent Lie algebra, i.e. $[\mathfrak{g},[\mathfrak{g},\mathfrak{g}]]=0$, and let $2r$ denote the maximal coadjoint orbit dimension of $\mathfrak{g}$. Then $\mathfrak{g}^*$ is desingularizable if and only if $\mathfrak{g}$ contains an abelian subalgebra of dimension $n-r$.
\end{corollary*}
The paper is organized as follows. Section 2 establishes the framework of desingularizable Poisson manifolds, including foundational examples and basic structural properties before discussing local desingularizability, and proving the semisimple obstruction. Section 3 turns to the linear case; Subsection 3.1 provides the necessity direction of the main theorem, and Subsection 3.2 provides sufficiency. Subsection 3.3 then gives proofs of the corollaries listed above.
\begin{Acknow}
    We would like to thank Rui Loja Fernandes for conversations inspiring this work, as well as Wilmer Smilde and Deniz Ozbay for conversations surrounding the work during Gone Fishing 2026 at the CRM in Montreal, and the anonymous referee for various suggestions that greatly improved the article's clarity and presentation.
\end{Acknow}
\section{Desingularization}
\subsection{Background}
Throughout this paper, $M$ will denote a connected, smooth manifold without boundary of dimension $n$ unless otherwise specified. For a vector bundle $E \ra M$, $\Gamma(E)$ will denote the space of global smooth sections of $E$.
\begin{definition}
	A \emph{Poisson algebroid} is a Lie algebroid $(\A, \rho, [ \cdot, \cdot])$ over $M$, with a choice of $\pi_\A \in \Gamma(\bigwedge^2 \A)$ satisfying $[\pi_\A, \pi_\A]=0$, where the bracket is the Nijenhuis-Schouten extension of the bracket on $\Gamma(\A)$.
\end{definition}
Often, we will simply write $(\A,\pi_\A)$ when the anchor map and bracket are clear from context. In the case that $\A$ is chosen to be the tangent bundle to $M$, this exactly recovers the standard notion of a Poisson bivector, i.e. an element $\pi \in \Gamma(\bigwedge^2 TM)$ such that
\begin{equation*}
	\pi(df,dg) = \{f,g\}, \qquad \forall f,g \in C^\infty(M)
\end{equation*}
where $[\pi,\pi]=0$ is equivalent to the Jacobi identity of $\{ \cdot , \cdot \}$. In the case that $\pi_\A$ is nondegenerate, the map $\pi_\A^\sharp$, up to sign, is an isomorphism of cochain complexes
\begin{equation*}
    \wedge^k\pi_\A^\sharp:\left( \Gamma(\wedge^k\A^*) , d_A\right) \longrightarrow \left(\Gamma\left(\wedge^k\A\right) , d_{\pi_\A}:=[\pi_\A,-]\right),
\end{equation*}
which induces an isomorphism in all degrees between the algebroid Poisson cohomology and the algebroid de Rham cohomology. This is completely analogous to the case of $\A=TM$, where it recovers the classical fact that the Poisson cohomology of a symplectic manifold agrees with the de Rham cohomology.
In this case of invertibility, there is a well-defined dual element to $\pi_\A^\sharp$, denoted $\omega_\A=\pi_\A^{-1}$, which plays the role of a symplectic structure.
\begin{definition}
	A \emph{symplectic Lie algebroid} is a Lie algebroid $\A$, with a choice of closed, nondegenerate element $\omega_\A \in \Gamma(\bigwedge^2\A^*)$.
\end{definition}

\begin{prop}\label{PoissonIFFclosed}
	If $(\A,\pi_\A)$ is a Poisson algebroid with invertible $\pi_\A$, then $\pi_\A^{-1}=:\omega_\A \in \Gamma(\bigwedge^2 \A^*)$ satisfies $d_\A  \omega_\A=0$.
\end{prop}
\begin{proof}
	See \cite{RANKIN2026102324}*{Proposition 72}.
\end{proof}
In light of this, we will call a Poisson algebroid \emph{symplectic} if $\pi_\A$ is nondegenerate, and refer  freely to $\omega_\A=\pi_\A^{-1}$ as described above. It was originally observed in \cite{nest1999deformationssymplecticliealgebroids} that any symplectic Lie algebroid $(\A,\rho,\pi_\A)$ over a manifold $M$ induces a Poisson structure on $M$ via the formula
\begin{equation*}
	\{f,g\}:= (\wedge^2 \rho (\pi_\A))(df,dg).
\end{equation*}
The question driving this article is understanding which Poisson manifolds arise this way if one insists that $\rho$ is generically injective.
\subsection{Desingularization Algebroids}

\begin{definition}
	A Lie algebroid $(\A,\rho, [\cdot,\cdot])$ over $M$ is said to be \emph{almost-injective} if the anchor map $\rho:\A \ra TM$ is injective on an open dense set of $M$, or equivalently, if the induced map on sections (which we denote with the same name by abuse of notation) $\rho:\Gamma(\A) \ra \Gamma(TM)$ is injective.
\end{definition}

\begin{definition}
We say that a Poisson manifold $(M,\pi)$ is \emph{algebroid desingularizable} or \emph{desingularizable} if there exists a symplectic Lie algebroid $(\A,\rho,\pi_\A)$ over $M$ with an almost-injective anchor such that
\begin{equation*}
	\pi = \wedge^2 \rho (\pi_\A).
\end{equation*}
If such an algebroid exists, we refer to it as a \emph{desingularizing algebroid} for $M$\footnote{ Though such a Lie algebroid is necessarily integrable, see \cite{MR1973056}*{Corollary 5.9}, we will not discuss the integrating Lie groupoids in this article.}. If no such algebroid exists\footnote{There are topological obstructions to the existence of such algebroids, see \cite{MR4178876}.}, then we say that $(M,\pi)$ is \emph{non-desingularizable}.
\end{definition}
Symplectic manifolds are the prototypical example, where the algebroid can be taken to be the tangent bundle with the identity as the anchor map. In fact, one can readily check that if $(\A,\rho,\pi_\A)$ is a desingularizing algebroid for a Poisson manifold $(M,\pi)$, and $\pi$ is invertible, i.e. $M$ is symplectic, then $\A \cong TM$ as Lie algebroids.
\begin{remark}
    In this way, a Poisson manifold that is symplectic cannot be desingularized using this algebroid procedure in any meaningful way. This is in the same spirit as \cite{klaasse2017geometricstructuresliealgebroids}*{Proposition 5.6.12}, and the remark preceding it.
\end{remark}

\begin{remark}\label{ArankBound}
If $(M,\pi)$ is desingularized by $(\A,\rho,\pi_\A)$, then we have 
\begin{equation*}
    \mathrm{rank}(\pi_x^\sharp) \leq \mathrm{rank}(\A) \qquad \forall x \in M,
\end{equation*}
with equality on the open dense set where $\rho_x$ is injective. Moreover, if $\mathcal{L}$ denotes the set of leaves of the symplectic foliation of $M$, then we have that
\begin{equation*}
    \mathrm{rank}(\A) = \max_{L \in \mathcal{L}}\{\dim(L)\}.
\end{equation*}
\end{remark}
\begin{exmp}
    A first example of a non-symplectic desingularizable Poisson manifold is that of a \emph{regular} Poisson manifold. In this case, if  $\mathcal{F}$ denotes the symplectic foliation, then $T\mathcal{F}$ is a desingularizing algebroid, with the leafwise symplectic form being the algebroid symplectic form.
\end{exmp}
\begin{exmp}
	$b$-manifolds are an example of where one needs the \emph{almost} in almost-injective. A $b$-Poisson manifold carries the $b$-tangent bundle ${}^bTM$, a Lie algebroid whose anchor is the inclusion of $b$-vector fields into $TM$, which is injective away from the critical hypersurface $Z$, hence almost-injective. The original Poisson structure lifts to an invertible section ${}^b\pi \in \Gamma(\bigwedge^2{}^bTM)$, see \cite{Guillemin_2014}, and desingularizes the structure.
\end{exmp}
More examples of desingularizable Poisson structures can be constructed from the following observations, which are easily verified:
\begin{remark}\label{TrivialIsDesing}
    Let $(M,\pi),(N,\pi') $ be Poisson manifolds.
    \begin{enumerate}
        \item If $\pi \equiv 0$, then $(M,\pi)$ is desingularizable.
        \item If $(M,\pi)$ and $(N,\pi')$ are Poisson-diffeomorphic and $(M,\pi)$ is desingularizable, then so too is $(N,\pi')$.
        \item If both $(M,\pi)$ and $(N,\pi')$ are desingularizable, so too is their product.
    \end{enumerate}
\end{remark}

\begin{remark}\label{AutomaticPoisson}
	If $(\A,\rho,[\cdot, \cdot])$ is an almost-injective Lie algebroid over a Poisson manifold $(M,\pi)$, then any $\pi_\A \in \Gamma(\bigwedge^2 \A)$ such that $\wedge^2 \rho(\pi_\A)=\pi$ is automatically Poisson, as we have that
	\begin{equation*}
		\rho([\pi_\A,\pi_\A]) = [\rho(\pi_\A),\rho(\pi_\A)] = [\pi,\pi]=0,
	\end{equation*}
where the first equality follows from the fact that $\rho$ is a morphism of Lie algebroids, thus $[\pi_\A,\pi_\A]=0$ by injectivity of $\rho$ on sections.
\end{remark}
 In the following section we will characterize which linear Poisson structures are desingularizable, but we can see now that not all structures are:
\begin{exmp}[Non-example]\label{NonexampleBump}
	Consider $\R^2$ with the Poisson bivector $f \p_x \wedge \p_y$ where $f$ is a bump function supported on the unit disc. By Remark \ref{ArankBound}, a desingularizing algebroid would force $\mathrm{rank}(\pi_x^\sharp)=\mathrm{rank}(\A)$ to be constant on the open dense set on which the anchor injects; however $\mathrm{rank}(\pi_x^\sharp)=2$ on the open disc, and $0$ outside it, thus the rank is non-constant on any open dense set of $\R^2$.
\end{exmp}

\subsection{Local Criteria}
Though not all Poisson structures are globally desingularizable, one can still ask for a local desingularization.
\begin{definition}\label{LocalDesingDef}
    We say that $(M,\pi)$ is \emph{locally desingularizable at $p$} if there exists an open set $U \ni p$ and a symplectic Lie algebroid $(\A,\rho,\pi_\A)$ over $U$ with almost-injective anchor such that $\wedge^2\rho(\pi_\A)=\pi|_U$.
\end{definition}
Note that if $(M,\pi)$ is desingularizable, then it is locally desingularizable at every point. Locally desingularizing algebroids need not be isomorphic, or even share the same rank:
\begin{exmp}
    Let $(M,\pi)$ be as in Example \ref{NonexampleBump}. If we choose $p$ outside the unit disc, then there is an open set $U$ containing $p$ on which $f$ identically vanishes, thus the Poisson structure is given by the zero bivector. By Remark \ref{TrivialIsDesing}, the zero algebroid is a local desingularization on $U$. If we choose $p$ to be the origin, then there is an open neighborhood $U$ on which $f$ is nonvanishing, thus the bivector is invertible on $U$, hence $T\R^2|_U$ with the same bivector is a local desingularization.
\end{exmp}
Since global desingularizability implies local desingularizability, we investigate local obstructions.
\begin{definition}
    Let $(M,\pi)$ be a Poisson manifold. We say that $x \in M$ is an \emph{isotropy point} if $\pi_x$ vanishes.
\end{definition}
 If $x$ is an isotropy point, $T^*_xM$ inherits a bracket from the Poisson bracket on $T^*M$, explicitly given by $[d_xf,d_xg]_\pi=d_x\{f,g\}$ \footnote{This is the isotropy Lie algebra of the cotangent Lie algebroid's anchor}. We will refer to $T_x^*M$ equipped with this bracket as $\mathfrak{g}_x$. Note that given a symplectic Lie algebroid $(\A,\rho, \pi_\A)$ inducing $\pi$, we can define $\Phi = \pi_\A^\sharp \circ \rho^*$. At an isotropy point $x \in M$, this gives a well-defined map given by
\begin{align*}
    \Phi_x:\mathfrak{g}_x &\ra \ker(\rho_x),
\end{align*}
as $\rho_x(\Phi_x(\alpha))=\pi^\sharp_x(\alpha) =0$. Note that $\ker(\rho_x)$ is also a Lie algebra with the isotropy Lie bracket.
\begin{lemma}\label{AbelianIdeal} 
If $x \in M$ is an isotropy point, then $\ker(\Phi_x)$ is an abelian ideal of $\mathfrak{g}_x$.
\end{lemma}
\begin{proof}
    We first claim that $\Phi_x$ is a morphism of Lie algebras. Toward this, let $V$ be an open set containing $x$, and let $\alpha,\beta \in \Gamma(T^*V)$. Since $\pi^\sharp$ is a morphism of Lie algebroids, on the level of sections we have that
    \begin{equation*}
        \rho(\Phi([\alpha,\beta]_\pi))=\pi^\sharp([\alpha,\beta]_\pi) = [\pi^\sharp(\alpha),\pi^\sharp(\beta)] = [\rho(\Phi(\alpha)), \rho(\Phi(\beta))] = \rho([\Phi(\alpha),\Phi(\beta)]).
    \end{equation*}
    Since $\rho$ is injective on sections, $\Phi$ preserves brackets. Isotropy bracket values are independent of local extension, so choosing constant local extensions of $\alpha_x,\beta_x$ gives $\Phi_x([\alpha_x,\beta_x]_\pi) = [\Phi_x(\alpha_x),\Phi_x(\beta_x)]$. From this, we see that $\Phi_x$ is a morphism of Lie algebras, hence $\ker(\Phi_x)$ is an ideal of $\mathfrak{g}_x$. Note that $(\pi_\A^\sharp)_x$ is invertible, so $\ker(\Phi_x)=\ker(\rho^*_x)$, so it suffices to check that the bracket vanishes for arbitrary elements $u_0,v_0 \in \ker(\rho_x^*)$ to show $\ker(\Phi_x)$ is abelian. Fix such $u_0,v_0$, and a chart $U$ centered at $x$. Let $\ell_u,\ell_v$ be the linear functions in $U$ with $d_p\ell_u=u_0,d_p\ell_v=v_0$ for all $p \in U$, i.e. whose differentials are constant $1$-forms in $U$. Define
    \begin{equation*}
        f:=\{\ell_u,\ell_v\} = \pi(d\ell_u,d\ell_v) = \pi_\A(\rho^*(d\ell_u),\rho^*(d\ell_v)).
    \end{equation*}
    After possibly shrinking $U$, choose a local frame of $\A^*|_U$, $\{e^1, \dots ,e^k\}$, and let $\phi^\alpha \in C^\infty(U)$ be such that 
    \begin{equation*}
        \rho_p^*(d\ell_u) = \sum_\alpha \phi^\alpha(p)e_p^\alpha \qquad \forall p \in U.
    \end{equation*}
    Since $\rho_x^*(d \ell_u)=0$, we must have that $\phi^\alpha(x)=0$ for all $\alpha$, hence by Hadamard's lemma there are smooth functions $h_i^\alpha \in C^\infty(U)$ such that
    \begin{equation*}
        \phi^\alpha(p) = \sum_i p_i h_i^\alpha(p) \qquad \forall  p \in U.
    \end{equation*}
    Setting $g_i^u = \sum_\alpha h_i^\alpha e^\alpha$, we have that
    \begin{equation*}
        \rho_p^*(d\ell_u) = \sum_\alpha\left(\sum_i p_ih_i^\alpha(p)\right)e^\alpha_p = \sum_i p_i g_i^u(p) \qquad \forall p \in U.
    \end{equation*}
    We analogously find $g_j^v(p)$ such that
    \begin{equation*}
        f(p) =(\pi_\A)_p(\rho_p^*(d\ell_u),\rho_p^*(d\ell_v)) = \sum_{i,j}p_ip_j(\pi_A)_p(g_i^u(p),g_j^v(p)) \qquad \forall p \in U,
    \end{equation*}
    thus $f$ vanishes to order $2$ at $x$. From this, we must have that $d_xf$ vanishes, and so
    \begin{equation*}
        0 = d_xf = d_x\{\ell_u,\ell_v\} = [d_x\ell_u,d_x\ell_v]_\pi = [u_0,v_0],
    \end{equation*}
    showing $\ker(\Phi_x)$ is abelian.
\end{proof}
We will show that $\ker(\Phi_x)$ is nonzero; first we need an linear algebra lemma.
\begin{lemma}\label{isotropic-rank}
Let $B$ be a bilinear form on a real vector space $V$ of dimension $n$, and
let $W \subseteq V$ be an isotropic subspace, i.e., for all $w_1,w_2 \in W$, we have that $B(w_1,w_2) = 0$. Then $\operatorname{rank}(B) \leq 2(n-\dim(W))$.
\end{lemma}

\begin{proof}
Let $\phi\colon V \to V^*$ be the linear map $\phi(v) = B(v,\,\cdot\,)$, so that
$\operatorname{rank}(B) = \dim\operatorname{im}(\phi)$ and $\dim\ker(\phi) = n - \operatorname{rank}(B)$.
Set $ \operatorname{im}(\phi|_W) =:W'\subseteq V^*$. Since
$W$ is isotropic, $(\phi(w_1))(w_2)=0$ for all
$w_i \in W$, so $W' \subset W^\circ$ and hence
\begin{equation*}
    \dim(W') \leq \dim(W^\circ)= n - \dim(W).
\end{equation*}
On the other hand $\ker(\phi|_W) = W \cap \ker(\phi) \subseteq \ker(\phi)$, so
\begin{equation*}
    \dim(W') = \dim(W) - \dim \ker(\phi|_W) \geq \dim( W) - \dim \ker(\phi) = \dim(W) - (n-\mathrm{rank}(B)).
\end{equation*}
Combining the two inequalities, we see
\begin{equation*}
    \dim (W) - (n-\mathrm{rank}(B)) \leq n-\dim(W),
\end{equation*}
thus $\mathrm{rank}(B) \leq 2(n- \dim(W))$.
\end{proof}

\begin{prop}\label{abelianidealrankbound}
    Let $\mathfrak{a}$ be an abelian subalgebra of an $n$-dimensional Lie algebra $\mathfrak{g}$. Then $\dim(\mathfrak{a}) \leq n-r$, where $2r=\max_{\mathcal{O}\in \mathcal{L}}\{\dim(\mathcal{O})\}$ and $\mathcal{L}$ denotes the collection of coadjoint orbits.
\end{prop}
\begin{proof}
    Fix a point $\xi \in \mathfrak{g}^*$, and define the bilinear form $B_\xi$ on $\mathfrak{g}$ via $B_\xi(v,w) = \xi([v,w])$ for $v,w \in \mathfrak{g}$. The tangent space to the coadjoint orbit through $\xi$ is exactly the image of $\mathrm{ad}_v^*\xi$, and so $\mathrm{rank}(B_\xi) = \dim(\mathcal{O}_\xi)$. Now suppose $\xi$ is such that $\mathrm{rank}(B_\xi)$ is maximal, i.e. $\mathrm{rank}(B_\xi) = 2r$. Since $\mathfrak{a}$ is abelian, it is isotropic for $B_\xi$. Using Lemma \ref{isotropic-rank}, we then have that $2r = \mathrm{rank}(B_\xi) \leq 2(n- \dim(\mathfrak{a}))$.
\end{proof}

\begin{prop}\label{DoubleBound}
 Let $(M,\pi)$ be an $n$-dimensional Poisson manifold, and $x \in M$ be an isotropy point. If $\A$ locally desingularizes $\pi$ on an open set $U \ni x$, we have that
 \begin{equation*}
     n - \frac{1}{2}\mathrm{rank}(\A) \leq \dim(\ker(\Phi_x)) \leq n-r_x,
 \end{equation*}
 where $2r_x$ is the maximal coadjoint orbit dimension of the Lie algebra $\mathfrak{g}_x$. Moreover, we have that $\dim(\ker(\Phi_x)) \geq \frac{n}{2}$, thus $\ker(\Phi_x)$ is a nonzero abelian ideal.
\end{prop}
\begin{proof}
    We first claim that $\im(\rho_x^*) \subset \A_x^*$ is isotropic with respect to $(\pi_\A)_x$. Toward this, fix $\alpha,\beta \in T_x^*M$, and letting $\langle \cdot,\cdot \rangle$ denote the dual pairing, we have
    \begin{equation*}
        (\pi_\A)_x(\rho_x^*(\alpha),\rho_x^*(\beta)) = \langle  (\rho \circ \pi_\A^\sharp \circ \rho^*)_x(\alpha),\beta \rangle = \langle \pi_x^\sharp(\alpha),\beta \rangle = 0,
    \end{equation*}
    where the last equality follows from $x$ being an isotropy point. Since $\im(\rho_x^*)$ is isotropic, $\mathrm{rank}(\rho_x^*) \leq \frac{1}{2}\mathrm{rank}(\A)$, and as $(\pi_\A)_x$ is invertible, $\mathrm{rank}(\Phi_x) = \mathrm{rank}(\rho_x^*)$. Using rank-nullity, we then have that
    \begin{equation*}
        n = \dim(\ker(\Phi_x)) + \mathrm{rank}(\Phi_x) = \dim(\ker(\Phi_x)) + \mathrm{rank}(\rho_x^*) \leq  \dim(\ker(\Phi_x)) + \frac{1}{2}\mathrm{rank}(\A),
    \end{equation*}
    which upon subtracting gives the lower bound. On the other hand, by Proposition \ref{abelianidealrankbound}, we must have that $\dim(\ker(\Phi_x)) \leq n-r_x$ since $\ker(\Phi_x)$ is an abelian ideal, establishing the claimed inequalities. To see that $\ker(\Phi_x)$ is nonzero, note that by Remark \ref{ArankBound}, we have that $\mathrm{rank}(\A)=\max_{p \in U}\mathrm{rank}(\pi_p^\sharp) \leq n$. Rearranging the established lower bound we have that
    \begin{align*}
        n \leq \dim(\ker(\Phi_x)) + \frac{1}{2}\mathrm{rank}(\A),
    \end{align*}
    and substituting $\mathrm{rank}(\A) \leq n$ gives
    \begin{equation*}
        n \leq \dim(\ker(\Phi_x)) + \frac{1}{2}\mathrm{rank}(\A) \leq \dim(\ker(\Phi_x)) +\frac{n}{2}.
    \end{equation*}
    Subtracting then gives us that $\dim(\ker(\Phi_x)) \geq \frac{n}{2}$, in particular $\ker(\Phi_x) \neq  \{0\}$.
\end{proof}
\begin{corollary}\label{isotropyPointNonzeroIdeal}
     If $(M,\pi)$ is $n$-dimensional and locally desingularizable at an isotropy point $x \in M$, then $\mathfrak{g}_x$ contains an abelian ideal $\mathfrak{a}_x$ satisfying $\dim(\mathfrak{a}_x) \geq\frac{n}{2}$. In particular $\mathfrak{a}_x$ is nonzero.
\end{corollary}
\begin{corollary}
    Let $(M,\pi)$ be a Poisson manifold. If there exists an isotropy point $x$ with $\mathfrak{g}_x$ semisimple, then $(M,\pi)$ is not locally desingularizable at $x$, hence not globally desingularizable.
\end{corollary}
\begin{proof}
    Suppose a locally desingularizing algebroid $(\A,\rho,\pi_\A)$ exists. By Corollary \ref{isotropyPointNonzeroIdeal} $\ker(\Phi_x)$ is a nonzero abelian ideal of $\mathfrak{g}_x$. If $\mathfrak{g}_x$ is semisimple, no such ideal can exist, a contradiction.
\end{proof}

\section{Linear Poisson Structures}
Given a real finite-dimensional Lie algebra $\mathfrak{g}$ one can associate a canonical Poisson structure to the dual $\mathfrak{g}^*$, by defining the Poisson bracket for $f,g \in C^\infty(\mathfrak{g}^*)$ as
\begin{equation*}
    \{f,g\}(\xi) := \langle \xi, [d_\xi f, d_\xi g] \rangle
\end{equation*}
for any $\xi \in \mathfrak{g}^*$. This is commonly referred to as the Kirillov-Kostant-Souriau or KKS Poisson structure, and the symplectic leaves  are the coadjoint orbits. The goal of this section is to prove the following theorem.
\begin{theorem}\label{MainThm}
    Let $\mathfrak{g}$ be a real finite-dimensional Lie algebra of dimension $n$, with $2r=\max_{\mathcal{O} \in \mathcal{L}}\{\dim(\mathcal{O})\}$ where $\mathcal{L}$ is the set of coadjoint orbits. Then $\mathfrak{g}^*$ is desingularizable if and only if there exists an abelian ideal $\mathfrak{a}$ of $\mathfrak{g}$ of dimension $n-r$.
\end{theorem}
\begin{remark}\label{WilmerConstruction}
    In contrast, if one drops the insistence of $\rho$ being almost-injective, one can always find a symplectic Lie algebroid inducing a given linear Poisson structure, see \cite{MR4421028}*{Section 2.3}.
\end{remark}
\subsection{Necessary Condition}
We will use the identification between $(T_0^*\mathfrak{g}^*, [\cdot,\cdot]_\pi)$ and $(\mathfrak{g}, [\cdot , \cdot])$.
Note that for any linear Poisson structure, $0 \in \mathfrak{g}^*$ is always an isotropy point, and so by Lemma \ref{AbelianIdeal}, we have that $\ker(\Phi_0)$ is an abelian ideal of $T_0^*\mathfrak{g}^* \cong \mathfrak{g}$. We now identify $\dim(\ker(\Phi_0))$. 
\begin{prop}\label{kerphidim}
    $\dim(\ker(\Phi_0)) = n-r$.
\end{prop}
\begin{proof}
One can identify the isotropy Lie algebra $ \mathfrak{g}_0\cong T_0^*\mathfrak{g}^* \cong \mathfrak{g}$, thus $r_0=r$. Proposition \ref{DoubleBound}, applied at $0$, gives
\begin{equation*}
     n-\frac{1}{2}\mathrm{rank}(A) \leq \dim\ker(\Phi_0) \leq n-r.
\end{equation*}
Since $A$ desingularizes $\mathfrak g^*$, Remark \ref{ArankBound} implies that
$\mathrm{rank}(A)=\max_{\mathcal{O} \in \mathcal{L}}\dim \mathcal{O}$, where $\mathcal{L}$ is the set of coadjoint orbits, whose maximal dimension is $2r$ by definition; hence
$\mathrm{rank}(A)=2r$. The lower bound above is then
$n-r\le\dim\ker(\Phi_0)$, thus $\dim\ker(\Phi_0)=n-r$.
\end{proof}
\begin{theorem}\label{onlyifdirection}
    If $\mathfrak{g}^*$ is desingularized by $(\A,\rho,\pi_\A)$, then $\mathfrak{g}$ contains an abelian ideal $\mathfrak{a}$ of dimension $n-r$.
\end{theorem}
\begin{proof}
    Combining Lemma \ref{AbelianIdeal} and Proposition \ref{kerphidim}, we have that $\ker(\Phi_0)$ is an abelian ideal of $T_0^*\mathfrak{g}^*$ of dimension $n-r$. Identifying $T_0^*\mathfrak{g}^*$ with $\mathfrak{g}$, we obtain an abelian ideal $\mathfrak{a}$ of dimension $n-r$.
\end{proof}

\subsection{Sufficient Condition}
Throughout this subsection, fix an $n$-dimensional Lie algebra $\mathfrak{g}$, and assume there exists an abelian ideal $\mathfrak{a} \subset \mathfrak{g}$ of dimension $n-r$, where again, $2r$ is the maximal coadjoint orbit dimension of $\mathfrak{g}$. Given this setup, one can consider, up to sign, the dual of the action of $\mathfrak{g}/\mathfrak{a}$ on $\mathfrak{a}$, which we realize as follows:
\begin{prop}
    For $u \in \mathfrak{g}$, let $M_u \in \End(\mathfrak{a}^*)$ be given by $M_u(\eta) := \eta \circ \ad_u|_\mathfrak{a}$, and let
    $\theta_u \in \Gamma(T\mathfrak{a}^*)$ be the linear vector field $\theta_u(\eta) := M_u(\eta) \in \mathfrak{a}^* \cong T_\eta \mathfrak{a}^*$. Then the mapping
    \begin{align*}
        \Theta : \mathfrak{g}/\mathfrak{a} &\ra \Gamma(T\mathfrak{a}^*) \\
        [u] &\mapsto \theta_u
    \end{align*}
    is a well-defined morphism of Lie algebras.
\end{prop}
\begin{proof}
    The fact that this is a well-defined linear map follows from linearity of $\ad$ and the fact that for any $a \in\mathfrak{a}$, $\ad_a|_\mathfrak{a}=0$. To see it preserves brackets, we first compute the commutator in $\End(\mathfrak{a}^*)$. For $u,w \in \mathfrak{g}$ and $\eta \in \mathfrak{a}^*$, using the Jacobi identity we see that
    \begin{equation*}
        [M_u,M_w](\eta) = \eta \circ \big(\ad_w \ad_u - \ad_u \ad_w\big)\big|_\mathfrak{a} = \eta \circ \ad_{[w,u]}|_\mathfrak{a} = M_{[w,u]}(\eta),
    \end{equation*}
    that is, $[M_u,M_w] = M_{[w,u]} = -M_{[u,w]}$, so $u \mapsto M_u$ is an anti-homomorphism.
    Next, for $S,T \in \End(\mathfrak{a}^*)$ let $L_S$ denote the linear vector field $\eta \mapsto S\eta$. Recall that in this setting, we have that $[L_S,L_T]=L_{[T,S]}$, and so noting that $\theta_u = L_{M_u}$ by definition, we have
    \begin{equation*}
         [\theta_u,\theta_w] = [L_{M_u},L_{M_w}] = L_{[M_w,M_u]} = L_{M_{[u,w]}} = \theta_{[u,w]}.
    \end{equation*}
\end{proof}
As this is a valid action, we can form the action algebroid $\mathfrak{g}/\mathfrak{a}\ltimes \mathfrak{a}^*$ over $\mathfrak{a}^*$. As a bundle, $\mathfrak{g}/\mathfrak{a}\ltimes \mathfrak{a}^*$ is the trivial bundle $\mathfrak{a}^* \times \mathfrak{g}/\mathfrak{a} \ra \mathfrak{a}^*$. The anchor map is given by 
\begin{align*}
    \rho_{\mathfrak{a}^*}: \mathfrak{a}^* \times \mathfrak{g}/\mathfrak{a} &\longrightarrow T\mathfrak{a}^* \\
    (\eta, [u]) &\longmapsto \Theta([u])(\eta) = \eta \circ \ad_u|_\mathfrak{a},
\end{align*}
and the bracket on sections, which are identified with smooth functions $\sigma:\mathfrak{a}^* \ra \mathfrak{g}/\mathfrak{a}$, is given by
\begin{equation*}
    [\sigma, \tau]:=[\sigma,\tau]_{\mathfrak{g}/\mathfrak{a}} + \rho_{\mathfrak{a}^*}(\sigma)\tau - \rho_{\mathfrak{a}^*}(\tau)\sigma,
\end{equation*}
where the bracket on the right-hand side is the pointwise bracket in $\mathfrak{g}/\mathfrak{a}$. Now, let $\mathrm{res}_\mathfrak{a}:\mathfrak{g}^* \ra \mathfrak{a}^*$ denote the restriction map.
\begin{lemma} \label{WegotanAlgebroid}
The vector bundle $\A:= \mathrm{res}_\mathfrak{a}^!(\mathfrak{g}/\mathfrak{a}\ltimes \mathfrak{a}^*)$, the Higgins-Mackenzie pullback of the algebroid $\mathfrak{g}/\mathfrak{a}\ltimes \mathfrak{a}^*$ along $\mathrm{res}_\mathfrak{a}$, is a Lie algebroid over $\mathfrak{g}^*$ of rank $2r$.
\end{lemma}
\begin{proof}
    The map $\mathrm{res}_\mathfrak{a}:\mathfrak{g}^* \ra \mathfrak{a}^*$ is a surjective submersion, thus is transverse to $\rho_{\mathfrak{a}^*}$, and so the pullback is a well-defined Lie algebroid over $\mathfrak{g}^*$, see \cite{MR1037400}. The rank of $\A$ follows from the formula
\begin{equation*}
    \mathrm{rank}(\mathrm{res}_\mathfrak{a}^!(\mathfrak{g}/\mathfrak{a}\ltimes \mathfrak{a}^*)) = \mathrm{rank}(\mathfrak{g}/\mathfrak{a}\ltimes \mathfrak{a}^*) - \dim(\mathfrak{a}^*) + \dim(\mathfrak{g}^*).
\end{equation*}
\end{proof}
Explicitly, at a point $\xi \in \mathfrak{g}^*$, we have
\begin{equation*}
    \A_\xi  = \left\{(X,[v]) \in T_\xi \mathfrak{g}^* \times \mathfrak{g}/\mathfrak{a}| d_\xi\mathrm{res}_\mathfrak{a}(X) = \rho_{\mathfrak{a}^*}( \mathrm{res}_\mathfrak{a}(\xi),v)\right\}.
\end{equation*}
The anchor is given by projection onto the first factor, i.e. $\rho(X,[v]) = X$, and the bracket on $\Gamma(\A)$ is given by
\begin{equation*}
    [(X,\sigma),(Y,\tau)]_\A = \big([X,Y],\, [\sigma,\tau]_{\mathfrak{g}/\mathfrak{a}} + X(\tau) - Y(\sigma)\big),
\end{equation*}
where $[X,Y]$ is the bracket of vector fields on $\mathfrak{g}^*$, the bracket $[\sigma,\tau]_{\mathfrak{g}/\mathfrak{a}}$ is again
taken pointwise. There is a convenient simplification of the pullback condition defining $\A$, which we will use to check if $(X,[v]) \in T_\xi \mathfrak{g}^* \times \mathfrak{g}/\mathfrak{a}$ lies in $\A_\xi$:
\begin{remark}\label{inAifPairing}
    Fix $\xi \in \mathfrak{g}^*$,  let $X \in T_\xi\mathfrak{g}^*$, $[v] \in \mathfrak{g}/\mathfrak{a}$, and let $v$ be any representative of $[v]$. Then $(X,[v]) \in \A_\xi$ if and only if
    \begin{equation*}
        \langle X, a \rangle = \xi([v,a]) \quad \forall a \in \mathfrak{a}.
    \end{equation*}
\end{remark}
\begin{proof}
    Since $\mathrm{res}_\mathfrak{a}$ is linear, we have that $d_\xi \mathrm{res}_\mathfrak{a}(X)(a) = \langle X,a \rangle$ for all $a \in \mathfrak{a}$. On the other hand, using that $\mathfrak{a}$ is an ideal we have that 
    \begin{equation*}
        \rho_{\mathfrak{a}^*}(\mathrm{res}_\mathfrak{a}(\xi), v)(a) = (\xi|_\mathfrak{a} \circ \ad_v|_\mathfrak{a})(a)= \xi([v,a]) \qquad \forall a \in \mathfrak{a}.
    \end{equation*}
\end{proof}

\begin{lemma}\label{SufficientAlmostInjective}
    Let $\pi$ denote the KKS Poisson structure on $\mathfrak{g}^*$. At any point $\xi$ for which $\pi_\xi$ is of maximal rank, the map $\rho_\xi:\A _\xi \ra T_\xi \mathfrak{g}^*$ is injective. In particular, $\A$ is almost-injective.
\end{lemma}
\begin{proof}
    Fix such a $\xi$, and suppose that $(X,[v]) \in \ker(\rho_\xi)$. We immediately have that $X=0$, and so using Remark \ref{inAifPairing}, we have that
    \begin{equation*}
        0 = \langle X,a\rangle = \xi([v,a]) \quad \forall a \in \mathfrak{a}.
    \end{equation*}
     If we define the bilinear form $B_\xi(x,y):=\xi([x,y])$ as in the proof of Proposition \ref{abelianidealrankbound}, then we claim $\mathfrak{a} + \R v$ is isotropic for $B_\xi$. Indeed, given $c,c' \in \R$ and $a,a' \in \mathfrak{a}$, we have that
     \begin{equation*}
         B_\xi(a+cv,a'+c'v) = B_\xi(a,a') + c'B_\xi(a,v) + cB_\xi(v,a')+ cc'B_\xi(v,v).
     \end{equation*}
     The first term vanishes as $\mathfrak{a}$ is abelian, the second two terms vanish by the above condition, and the last term vanishes by skew-symmetry of the bracket. Invoking Lemma \ref{isotropic-rank} along with the assumption that $\mathrm{rank}(B_\xi)=\mathrm{rank}(\pi_\xi)=2r$, we have that
     \begin{equation*}
         2r = \mathrm{rank}(B_\xi) \leq 2(n-\dim(\mathfrak{a} + \R v)).
     \end{equation*}
     Rearranging this gives us that
     \begin{equation*}
         \dim(\mathfrak{a} + \R v) \leq n-r = \dim(\mathfrak{a}),
     \end{equation*}
     thus $\mathfrak{a} = \mathfrak{a}+\R v$, so $v \in \mathfrak{a}$. Consequently, $[v]=0$, and so $\rho_\xi$ is injective. Since the set on which $\pi$ is of maximal rank is an open dense set, we have that $\rho$ is injective on an open dense set, and thus $\A$ is almost-injective.
\end{proof}

\begin{lemma}\label{SufficentAlmostSymplecticForm}
    At any point $\xi \in \mathfrak{g}^*$, the form
    \begin{equation*}
        \omega_\xi ((X,[u]),(Y,[v])) := \langle Y,u \rangle - \langle X,v \rangle +\xi([u,v]),
    \end{equation*}
    where $u,v$ are any representatives of $[u],[v]$, respectively, is a well-defined, smooth, skew-symmetric, nondegenerate, bilinear form on $\A_\xi$, that is, the $\omega_\xi$ define an algebroid almost symplectic form $\omega \in \Gamma(\bigwedge^2 \A^*)$.
\end{lemma}
\begin{proof}
   This form is clearly smooth, bilinear, and skew-symmetric provided that it is well-defined. To see it is well-defined, suppose that $u'=u+a_1$ and $v'=v+a_2$ for $a_i \in \mathfrak{a}$. Then using Remark \ref{inAifPairing}, we have that
   \begin{align*}
       \omega_\xi((X,[u']),(Y,[v']))-\omega_\xi((X,[u]),(Y,[v])) &=  \langle Y,a_1\rangle - \langle X,a_2 \rangle + \xi([u,a_2])+\xi([a_1,v]) + \xi([a_1,a_2]) \\
       &=\xi([v,a_1])-\xi([u,a_2]) +\xi([u,a_2])-\xi([v,a_1]) + \xi([a_1,a_2]) \\
       &=\xi([a_1,a_2]) \\
       &=0,
   \end{align*}
   where the last summand vanishes as $\mathfrak{a}$ is abelian, giving us that $\omega_\xi$ is well-defined. Toward nondegeneracy, suppose that $(X,[v]) \in \A_\xi$ is such that $\omega_\xi((X,[v]),e)=0$ for all $e \in \A_\xi$, and let $\mathfrak{a}^\circ$ denote the annihilator of $\mathfrak{a}$. For any $Y \in \mathfrak{a}^\circ$, Remark \ref{inAifPairing} gives us that $(Y,[0]) \in \A_\xi$. Using this, we have that 
   \begin{equation*}
       0 = \omega_\xi((X,[v]),(Y,[0])) = \langle Y, v \rangle.
   \end{equation*}
   As this holds for any $Y \in \mathfrak{a}^\circ$, we have that $v \in (\mathfrak{a}^\circ)^\circ = \mathfrak{a}$, thus $[v]=0$. For arbitrary $u \in \mathfrak{g}$, surjectivity of $d_\xi \mathrm{res}_\mathfrak{a}$ provides a $Z \in T_\xi \mathfrak{g}^*$ such that $d_\xi \mathrm{res}_\mathfrak{a}(Z) = \rho_{\mathfrak{a}^*}(\mathrm{res}_\mathfrak{a}(\xi),[u])$. This, together with the choice of $0$ as a representative of $[v]$, gives us that
   \begin{equation*}
       0 = \omega_\xi((X,[0]),(Z,[u])) = -\langle X,u \rangle.
   \end{equation*}
   As $u$ was arbitrary, we must have that $X=0$ by nondegeneracy of the pairing. Together we see that $\omega_\xi^\flat$ has trivial kernel, and so $\omega_\xi$ is nondegenerate.
\end{proof}
\begin{remark}
    If one removes the almost-injective requirement and considers the abelian ideal $\mathfrak{a}=\{0\}$, this process constructs $\A=T\mathfrak{g}^* \times\mathfrak{g}=p_*^!\mathfrak{g}$ where $p_*:\mathfrak{g}^* \ra \mathrm{pt}$, and $\omega$ defined above agrees with the canonical symplectic form of \cite[Definition 2.7]{MR4421028} as alluded to in Remark \ref{WilmerConstruction}.
\end{remark}
In light of nondegeneracy, we can define $\pi_\A \in \Gamma(\bigwedge^2\A)$ via enforcing $(\pi_\A^\sharp)_\xi = (\omega_\xi^\flat)^{-1}$.
\begin{lemma}\label{pi_Agoestopi}
    $\wedge^2 \rho(\pi_\A) = \pi$, where $\pi_\A$ is defined as above, and $\pi$ is the KKS Poisson bivector on $\mathfrak{g}^*$.
\end{lemma}
\begin{proof}
    Fix $\xi \in \mathfrak{g}^*$, $\alpha \in T_\xi^* \mathfrak{g}^* \cong \mathfrak{g}$, and define $e_\alpha:= (\pi_\xi^\sharp(\alpha),[\alpha]) \in T_\xi \mathfrak{g}^* \times \mathfrak{g}/\mathfrak{a}$. First, note that $e_\alpha \in A_\xi$ as $\langle \pi_\xi^\sharp(\alpha),a \rangle=\xi([\alpha,a])$ for any $a \in \mathfrak{a}$ by definition of the KKS bivector. For any $(Y,[v]) \in \A_\xi$, we have that
    \begin{align*}
        \omega_\xi(e_\alpha, (Y,[v])) &= \langle Y, \alpha \rangle - \langle \pi_\xi^\sharp(\alpha),v\rangle + \xi([\alpha,v]) =\langle Y, \alpha \rangle - \xi([\alpha,v]) + \xi([\alpha,v]) =\langle Y,\alpha\rangle.
    \end{align*}
    On the other hand, 
    \begin{equation*}
        (\rho^*\alpha)((Y,[v])) = \langle \rho(Y,[v]),\alpha \rangle = \langle Y,\alpha \rangle,
    \end{equation*}
    thus $\omega_\xi^\flat(e_\alpha)=(\rho^*)_\xi (\alpha)$, and so $e_\alpha = (\pi_\A^\sharp)_\xi((\rho^*)_\xi (\alpha))$. Applying this to arbitrary $\alpha,\beta \in T_\xi^* \mathfrak{g}^*$ yields
    \begin{align*}
       (\wedge^2 \rho(\pi_\A))_\xi(\alpha,\beta)=\langle (\pi_\A^\sharp)_\xi (\rho^*(\alpha)), \rho^*(\beta) \rangle =\langle (\pi_\xi^\sharp(\alpha),[\alpha]),\rho^*(\beta) \rangle =\langle \pi_\xi^\sharp(\alpha),\beta \rangle=\pi_\xi(\alpha,\beta),
    \end{align*}
    thus $(\wedge^2\rho)(\pi_\A) = \pi$.
\end{proof}
\begin{theorem}\label{DesingConstruction}
    The Lie algebroid $(\A=\mathrm{res}_\mathfrak{a}^!(\mathfrak{g}/\mathfrak{a} \ltimes \mathfrak{a}^*), \rho, [\cdot,\cdot])$, together with $\omega \in \Gamma(\bigwedge^2\A^*)$ as defined above, is a desingularizing algebroid for $(\mathfrak{g}^*,\pi)$.
\end{theorem}
\begin{proof}
 Lemma \ref{SufficientAlmostInjective} ensures that $\A$ is an almost-injective Lie algebroid. Almost-injectivity, together with Lemma \ref{pi_Agoestopi} and Remark \ref{AutomaticPoisson} ensures that $\pi_\A$ is Poisson. As such, Proposition \ref{PoissonIFFclosed} ensures that the almost symplectic form constructed in Lemma \ref{SufficentAlmostSymplecticForm} is indeed closed, i.e. symplectic. Thus, $(\A, \rho, \pi_\A)$ is a symplectic Lie algebroid satisfying $\wedge^2 \rho(\pi_\A)=\pi$ as required.
\end{proof}
\begin{remark}
    Throughout the construction of this algebroid, the dimension of $\mathfrak{a}$ is only used in two places. The first is in determining the rank of $\A$: for general $\mathfrak{a}$,  $\mathrm{rank}(\A) =2(\dim(\mathfrak{g})-\dim(\mathfrak{a}))$, and $\A$ still carries the form $\omega$ defined above. The second time the dimension of $\mathfrak{a}$ is used is in verifying almost-injectivity of the algebroid's anchor, which in turn is used to verify that the almost symplectic form is indeed closed. One can directly verify, however, that the form $\omega$ is closed independently of the dimension of $\mathfrak{a}$. As such, given any abelian ideal of a real finite-dimensional Lie algebra, one can construct in this way a symplectic Lie algebroid this way that induces the KKS Poisson structure. The condition  $\dim(\mathfrak{a})=n-r$ is needed precisely to obtain an \textit{almost-injective} algebroid.
\end{remark}
With this in place, we can prove Theorem \ref{MainThm}.
\begin{proof}[Proof of Theorem \ref{MainThm}]
    If $\mathfrak{g}^*$ is desingularizable, then Theorem \ref{onlyifdirection} provides such an abelian ideal $\mathfrak{a}$ of the correct dimension. If $\mathfrak{g}$ has an abelian ideal $\mathfrak{a}$ of the correct dimension, then Theorem \ref{DesingConstruction} constructs a desingularizing algebroid.
\end{proof}
\subsection{Corollaries}
\begin{corollary} \label{NoReductiveOnes}
    If $\mathfrak{g}$ is a real finite-dimensional, non-abelian reductive Lie algebra, then $\mathfrak{g}^*$ is not desingularizable.
\end{corollary}
\begin{proof}
    By Theorem \ref{MainThm}, it suffices to show that such a Lie algebra can have no abelian ideal of dimension $n-r$. Write $\mathfrak{g}=\mathfrak{s}\oplus \mathfrak{z}$ where $\mathfrak{s}$ is semisimple and $\mathfrak{z}$ is abelian, and suppose $\mathfrak{i}$ is an abelian ideal of $\mathfrak{g}$. We must have that $\mathfrak{i}=\mathfrak{i}\cap\mathfrak{z}$, as otherwise the projection of $\mathfrak{i}$ to $\mathfrak{s}$ would be a nontrivial abelian ideal of a semisimple Lie algebra, which cannot exist. From this, we see that $\dim \mathfrak{i} \leq \dim \mathfrak{z}$. By \cite{MR150240}*{Theorem 0.7}, we have that the maximal coadjoint orbit dimension of $\mathfrak{g}$ agrees with that of $\mathfrak{s}$, and is $\dim(\mathfrak{s})-\mathrm{rank}(\mathfrak{s})=2r$. We then have that
    \begin{equation*}
        n-r = (\dim(\mathfrak{s}) + \dim(\mathfrak{z}))-\frac{1}{2}(\dim(\mathfrak{s}) - \mathrm{rank}(\mathfrak{s})) = \dim(\mathfrak{z})+\frac{1}{2}(\dim(\mathfrak{s})+\mathrm{rank}(\mathfrak{s})).
    \end{equation*}
    Since $\mathfrak{s}$ is semisimple, we have that $\dim(\mathfrak{s}) \geq 3$ and $\mathrm{rank}(\mathfrak{s})\geq 1$, thus
    \begin{equation*}
        n-r \geq \dim(\mathfrak{z}) + 2 > \dim \mathfrak{z} \geq \dim \mathfrak{i},
    \end{equation*}
    thus no abelian ideal can reach the necessary dimension.
\end{proof}
\begin{corollary}
    If $\mathfrak{g}$ is a real Lie algebra of dimension $3$ or less, then $\mathfrak{g}^*$ is desingularizable if and only if $\mathfrak{g}$ is not semisimple.
\end{corollary}
\begin{proof}
    \textit{Dimension $1$:} The only Lie algebra up to isomorphism is the abelian one, whose dual is desingularizable by Remark \ref{TrivialIsDesing}.\newline
    \textit{Dimension $2$:} There are two Lie algebras up to isomorphism: the abelian one, whose dual is desingularizable again by Remark \ref{TrivialIsDesing}, and one given by $\mathrm{span}\{e_1,e_2\}$ with nonzero bracket $[e_1,e_2]=e_1$. Here, $n=2, r=1$ and $\mathfrak{a}:=\mathrm{span}\{e_1\}$ satisfies the criterion in Theorem \ref{MainThm}, hence its dual is desingularizable. \newline
    \textit{Dimension $3$:} Corollary \ref{NoReductiveOnes} rules out any semisimple Lie algebra, so suppose that $\mathfrak{g}$ is dimension $3$ and not semisimple. If $\mathfrak{g}$ is abelian, then its dual is desingularizable, so suppose that $\mathfrak{g}$ is not abelian. Since $\dim \mathfrak{g}=3$ and non-abelian, we have that $r=1$, so by Theorem \ref{MainThm}, we must show that there is an abelian ideal $\mathfrak{a} \subset \mathfrak{g}$ with $\dim \mathfrak{a}=2$. Since $\mathfrak{g}$ is not semisimple and of dimension $3$, we have that $[\mathfrak{g},\mathfrak{g}] \neq \mathfrak{g}$, thus $\dim([\mathfrak{g},\mathfrak{g}]) \in \{1,2\}$. 
    \begin{enumerate}
        \item Suppose that $\dim\mathfrak[\mathfrak{g},\mathfrak{g}]=1$. Fix a basis $\{v\}$ of $[\mathfrak{g},\mathfrak{g}]$. Note that $[\mathfrak{g},v]\subset [\mathfrak{g},[\mathfrak{g},\mathfrak{g}]] \subset [\mathfrak{g},\mathfrak{g}]$ so there must exist $\varphi \in \mathfrak{g}^*$ such that for all $x \in \mathfrak{g}$ we have that $[x,v]=\varphi(x)v$. By rank-nullity, $\dim(\ker(\varphi)) \geq 2$. Choose $w \in \ker(\varphi) \setminus \mathrm{span}\{v\}$, and define $\mathfrak{a}:=\mathrm{span}\{v,w\}$. We claim that $\mathfrak{a}$ is an abelian ideal. To see it is an ideal, note that for any $x \in \mathfrak{g}$, we have that $[x,w] \in [\mathfrak{g},\mathfrak{g}]$, thus there exists $\lambda \in \R$ such that $[x, w]=\lambda v \in \mathfrak{a}$, and likewise, $[x,v]=\varphi(x)v \in \mathfrak{a}$, hence $\mathfrak{a}$ is an ideal. To see it is abelian, note that the only possible nonzero bracket is $[w,v]=\varphi(w)v$, but $w \in \ker(\varphi)$, hence $[w,v]=0$, and so $\mathfrak{a}$ is abelian.
        \item Suppose that $\dim[\mathfrak{g},\mathfrak{g}]=2$. We claim that $\mathfrak{a}:=[\mathfrak{g},\mathfrak{g}]$ is a $2$-dimensional abelian ideal. The only part of this not clear is that it is abelian, which we check now. Choose a basis $\{a_1,a_2\}$ of $\mathfrak{a}$, and extend it to a basis $\{a_1,a_2,v\}$ of $\mathfrak{g}$. Since $\mathfrak{a}$ is an ideal, there are $p,q,t,s,\lambda,\mu \in \R$ such that
        \begin{equation*}
            [v,a_1]=pa_1+qa_2, \quad [v,a_2]=ta_1+sa_2, \quad [a_1,a_2]=\lambda a_1+\mu a_2.
        \end{equation*}
        From the Jacobi identity, we must have that
        \begin{equation*}
            t \mu = s\lambda \quad \text{ and }\quad  \lambda q = \mu p.
        \end{equation*}
        Suppose toward a contradiction that $(\lambda,\mu)\neq (0,0)$. If $\lambda \neq 0$, then we have that 
        \begin{align*}
            [v,a_1]&=pa_1+ \frac{\mu p}{\lambda}a_2 = \frac{p}{\lambda}[a_1,a_2], \\
            [v,a_2]&= ta_1+ \frac{t \mu}{\lambda}a_2 = \frac{t}{\lambda}[a_1,a_2],
        \end{align*}
        hence we must have that $[\mathfrak{g},\mathfrak{g}] \subset \mathrm{span}\{[a_1,a_2]\}$, which forces $\dim [\mathfrak{g},\mathfrak{g}] \leq 1$, a contradiction. If $\lambda=0$, so that $\mu \neq 0$, we then have that $t=0$ and $p=0$, thus
        \begin{equation*}
            [v,a_1] = qa_2, \quad [v,a_2]=sa_2, \quad [a_1,a_2]=\mu a_2,
        \end{equation*}
        again giving us that $\dim[\mathfrak{g},\mathfrak{g}]\leq 1$, a contradiction. We must then have that $(\lambda,\mu)=(0,0)$, and so $\mathfrak{a}$ is abelian.
    \end{enumerate}
\end{proof}

In the 2-step nilpotent case, Theorem \ref{MainThm} can be restated in the following manner:
\begin{corollary}
    Suppose that $\mathfrak{g}$ is a real $n$-dimensional $2$-step nilpotent Lie algebra, i.e. $[\mathfrak{g},[\mathfrak{g},\mathfrak{g}]]=0$, and let $2r$ denote the maximal coadjoint orbit dimension of $\mathfrak{g}$. Then $\mathfrak{g}^*$ is desingularizable if and only if $\mathfrak{g}$ contains an abelian subalgebra of dimension $n-r$.
\end{corollary}
\begin{proof}
  Suppose $\mathfrak{g}^*$ is desingularizable. Then Theorem \ref{MainThm} provides an abelian ideal of dimension $n-r$, which is of course an abelian subalgebra. Suppose now that an abelian subalgebra $\mathfrak{h}$ of dimension $n-r$ exists. Set $\mathfrak{a}:=\mathfrak{h}+Z(\mathfrak{g})$. We claim that $\mathfrak{a}$ is an abelian ideal of the correct dimension to satisfy Theorem \ref{MainThm}. It is clear that $\mathfrak{a}$ is abelian, to see it is an ideal, let $x \in \mathfrak{g}$ and $y \in \mathfrak{a}$. Then using the fact that $\mathfrak{g}$ is $2$-step nilpotent, we have that $[x,y] \in [\mathfrak{g,\mathfrak{g}}]\subset Z(\mathfrak{g}) \subset \mathfrak{a}$. Since $\mathfrak{h}\subset \mathfrak{a}$, we have that $\dim(\mathfrak{a}) \geq \dim(\mathfrak{h})=n-r$. On the other hand, using Proposition \ref{abelianidealrankbound}, we must have that $\dim(\mathfrak{a}) \leq n-r$, and so $\mathfrak{a}$ is an abelian ideal with $\dim(\mathfrak{a})=n-r$.
\end{proof}

\bibliographystyle{amsplain}
\bibliography{refs}
\end{document}